\documentclass[12pt,thmsa]{article}
\usepackage{amsfonts}
\usepackage{amsmath}
\usepackage{dsfont}
\usepackage{amsthm}
\usepackage{amssymb, graphics, setspace}
\usepackage[useregional]{datetime2}
\usepackage[english]{babel}
\usepackage{listings}

\usepackage[utf8]{inputenc}
\newcommand{\mathsym}[1]{{}}
\newcommand{\unicode}[1]{{}}

\newcommand{\rd}{\mathrm{d}}
\newtheorem{theorem}{Theorem}
\newtheorem{corollary}{Corollary}
\newtheorem{definition}{Definition}
\newtheorem{example}{Example}

\begin{document}

\title{Gauss maps of surfaces in 3-dimensional Heisenberg group}


\author{Christiam Figueroa \thanks{%
Principal professor of the Pontificia Universidad Cat\'{o}lica del Per\'{u}}}





\maketitle

\begin{abstract}
In this paper we study the Gauss map  of surfaces in 3-dimensional Heisenberg
group using the Gans model of the hyperbolic plane.  We establishes a relationship between the tension field of the Gauss map and  mean curvature
of a surface in $\mathcal{H}_{3}$.
\end{abstract}

\section{Introduction}

 It is well-known the classical Weierstrass representation formula describes
minimal surfaces in $\mathbb{R}^{3}$ in terms of their Gauss map. More generally, Kenmotsu \cite{kenmotsu1979weierstrass}
shows a representation formula for arbitrary surfaces in $\mathbb{R}
^{3}$ with non-vanishing mean curvature, which describes these surfaces in
terms of their Gauss map and mean curvature functions. Similar result have been obtained for minimal surface in the Heisenberg group, see \cite{daniel2011gauss} and \cite{figueroa2007gauss}.

Motivated by these results, we wish to investigate  a relationship between the tension field of the Gauss map and  mean curvature of a surface in $\mathcal{H}_{3}$.

We have organized the paper as follows. In section 2 we present the Gans model of the hyperbolic
plane where a geodesic is a line or a branch of a hyperbola.

Section 3 contains the basic Riemannian geometry of $\mathcal{H}_{3}$ equipped with a left-invariant metric.

In section 4 we review some of the standard facts on the nonparametric surface and vertical surface in $\mathcal{H}_{3}$. We
compute the coefficients of the first and second fundamental form of these surfaces.

Section 5 provides a detailed exposition of the Gauss map of  nonparametric and vertical surfaces. And establishes the
relation between the Gauss map of two  surfaces that  are related by an ambient isometry of $\mathcal{H}_{3}$.

In section 6 we describe the relationship  between  the tension field of the Gauss map and mean curvature  of a surface in $%
\mathcal{H}_{3}$. As a consequence, we describe all minimal surfaces in $\mathcal{H}_{3}$, such that
its Gauss map has rank 1. To end this section, we present a theorem that provides a characterization of a minimal surface such that its Gauss map is conformal.

Finally, we include an appendix with a Mathematica program that allows us to compute the tension field of the Gauss map.

\section{The Gans model}

This is a  model of the hyperbolic geometry, developed by David Gans, see \cite{gans1966new}. Unlike the other models,  utilizes  the entire plane. We shall present the basic concepts of this geometry, that is,  isometries group, the Riemannian connection and their geodesics.
Consider the Poincar\'{e} Disk
$$\mathbb{D}=\{(x,y):x^{2}+y^{2}<1\}$$
endowed with the metric
$$g(x,y)=\frac{4}{(1-x^{2}-y^{2})^{2}}(dx^{2}+dy^{2})$$

We  first define a diffeomorphism between  the Poincar\'{e} disk and the plane $\mathcal{P}:z=1$

Using the stereographic projection from the south pole $(0,0,-1)$ of the unite sphere, we have the following diffeomorphism  $\varphi$ between the upper hemisphere $S_{+}$ onto  the disk, $\mathbb{D}$,
$$\varphi(x,y,z)=(\frac{x}{z+1},\frac{y}{z+1}).
$$
In the same way, considering  the stereographic projection from the origin $(0,0,0)$ of the unite  sphere, we define a diffeomorphism $\psi$ of $S_{+}$, onto  the plane  $\mathcal{P}:z=1$,
\begin{equation}\label{psi}
  \psi(x,y,z)=(\frac{x}{z},\frac{y}{z},1).
\end{equation}
Then, $F(x,y)=\psi\circ\varphi^{-1}$ is a diffeomorphism  from the disk $\mathds{D}$ onto $\mathcal{P}$, where
\begin{equation}\label{diff}
  F(x,y)=(\frac{2x}{1-x^{2}-y^{2}},\frac{2y}{1-x^{2}-y^{2}},1)
\end{equation}
and the inverse is given by
$$F^{-1}(u,v,1)=(\frac{u}{1+\sqrt{1+u^{2}+v^{2}}},\frac{v}{1+\sqrt{1+u^{2}+v^{2}}})$$
Then the metric induced on  $\mathcal{P}$ by $F$ is given by
$$h(u,v)=\frac{(1+v^{2})du^{2}-2uvdudv+(1+u^{2})dv^{2}}{1+u^{2}+v^{2}}$$
The riemannian space $(\mathcal{P},h)$ is the Gans model of  the hyperbolic geometry.

\subsection{Isometries}

Consider the Poincar\'{e} disk  as the subset $\mathbb{D}=\{z\in  \mathbb{C}:|z|<1\}$
of the complex plane and the Gans model $\mathcal{P}=\{w:w\in \mathbb{C}\}$.
We know that the set of orientation-preserving isometries of the Poincar\'{e} Disk have the form,
$$\rho(z)=e^{i\theta}\frac{z-a}{1-\overline{a}z}, \qquad a\in\mathbb{D}.$$
And all  isometries of $\mathbb{D}$ are composed of $\rho$ with complex conjugation, that is reflection at the real axis.
Therefore, the isometry group of the Gans model is
$$Iso(\mathcal{P})=\{F\circ\rho\circ F^{-1}:\rho\in Iso(\mathbb{D})\},$$
where $F$ is as in $(\ref{diff})$. I shall highlight two cases:

If $\rho(z)=e^{i\theta}z$ , then  $F\circ\rho\circ F^{-1}(w)=e^{i\theta}w$, that is, a rotation about the origin $(0,0)$ is an isometry of the hyperbolic space $\mathcal{P}$.

On the other hand, if $\rho(z)=\overline{z}$, then  $F\circ\rho\circ F^{-1}(w)=\overline{w}$  is the
 reflection across the $u$ axis. Since rotation about the origin is an isometry, a reflection across the line $au+bv=0$  is an isometry too.

\subsection{Geodesics }

 The Christoffel symbols for the Gans Model are, (see Appendix \ref{appendix:a}):
\begin{equation*}
\begin{aligned}[t]
\text{$\Gamma^{1}_{11} $} &=& -\frac{u \left(v^2+1\right)}{u^2+v^2+1} \\
 \text{$\Gamma^{1}_{21} $} &=& \frac{u^2 v}{u^2+v^2+1} \\
 \text{$\Gamma^{1}_{22} $} &=& -\frac{u^3+u}{u^2+v^2+1} \\
\end{aligned}
\qquad
\qquad
\begin{aligned}[t]
\text{$\Gamma^{2}_{11} $} &=& -\frac{v^3+v}{u^2+v^2+1} \\
 \text{$\Gamma^{2}_{21} $} &=&\frac{u v^2}{u^2+v^2+1} \\
 \text{$\Gamma^{2}_{22} $} &=& -\frac{\left(u^2+1\right) v}{u^2+v^2+1} \\
\end{aligned}
\end{equation*}

Let  $\gamma(s)=(u(s),v(s))$ be a geodesic in this model such that $\|\gamma'(s)\|=1$, then the system of geodesics equations for $\gamma$ is given by
\begin{equation}
\left\{\begin{matrix}
  u'' -u&=& 0 \\
  v''-v &=& 0
\end{matrix}\right.
\end{equation}
Solving this system, we obtain that
$$u(t)=Ae^{t}+Be^{-t},\;\;\;\;\;\;\;\;v(t)=Ce^{t}+De^{-t},$$
with $A,B,C,D\in \mathbb{R}$. Varying these constants we get that the geodesics in the  Gans Model are  straight lines that pass through the origin and   a branch of a hyperbola with center at the origin $(0,0)$.

\section{The Geometry of the Heisenberg Group }
The 3-dimensional Heisenberg group $\mathcal{H}_{3}$ is a two-step nilpotent Lie group. It has the following standard representation
 in $GL_{3}(\mathbb{R})$
 $$\left[
      \begin{array}{ccc}
        1 & r & t \\
        0 & 1 & s \\
        0 & 0 & 1 \\
      \end{array}
    \right]
$$
with $r,s,t\in \mathbb{R}$.

In order to describe a left-invariant metric on $\mathcal{H}_{3}$, we note that the Lie algebra $\mathfrak{h}_{3}$ of $\mathcal{H}_{3}$ is given by the
matrices
$$A=\left[
      \begin{array}{ccc}
        0 & x & z \\
        0 & 0 & y \\
        0 & 0 & 0 \\
      \end{array}
    \right] $$
    with $x,y,z$ real. The exponential map $exp:\mathfrak{h}_{3}\rightarrow \mathcal{H}_{3}$ is a global diffeomorphism, and is given by
    $$exp(A)=I+A+\frac{A^{2}}{2}=\left[
                                   \begin{array}{ccc}
                                     1 & x & z+\frac{xy}{2} \\
                                     0 & 1 & y \\
                                     0 & 0 & 1 \\
                                   \end{array}
                                 \right].$$
    Using the exponential map as a global parametrization, with the identification of the Lie algebra $\mathfrak{h}_{3}$ with $\mathbb{R}^{3}$ given by
$$(x,y,z)\leftrightarrow \left[
      \begin{array}{ccc}
        0 & x & z \\
        0 & 0 & y \\
        0 & 0 & 0 \\
      \end{array}
    \right],$$
    the group structure of $\mathcal{H}_{3}$ is given by
\begin{equation}\label{pr}
(a,b,c)\ast (x,y,z)=(a+x,b+y,c+z+\frac{ay-bx}{2}).
\end{equation}
    From now on, modulo the identification given by $exp$, we consider $\mathcal{H}_{3}$ as $\mathbb{R}^{3}$ with the product given in (\ref{pr}). The Lie   algebra bracket, in terms of the canonical basis $\{e_{1},e_{2},e_{3}\}$ of $\mathbb{R}^{3}$, is given by
    $$[e_{1},e_{2}]= e_{3},  [e_{i},e_{3}]=0$$
    with $i=1,2,3.$ Now, using $\{e_{1},e_{2},e_{3}\}$ as the orthonormal frame at the identity, we have the following left-invariant metric $\rd s^{2}$ in $\mathcal{H}_{3},$
$$\rd s^{2}=\rd x^{2}+\rd y^{2}+(\frac{1}{2}y \rd x-\frac{1}{2}x\rd y+\rd z)^{2}.$$
And the basis of the orthonormal left-invariant vector fields is given by
    $$E_{1}=\frac{\partial}{\partial x}-\frac{y}{2}\frac{\partial}{\partial z},\;\;  E_{2}=\frac{\partial}{\partial x}+\frac{x}{2}\frac{\partial}{\partial z},\;\;
    E_{3}=\frac{\partial}{\partial z}.$$
Then the Riemann connection of $ds^{2}$, in terms of the basis $\{E_{i}\}$, is given by:
$$\begin{array}{ccccc}
  \nabla_{E_{1}}E_{2} & = & \frac{1}{2}E_{3} & = & -\nabla_{E_{2}}E_{1} \\
  \nabla_{E_{1}}E_{3} & = &  -\frac{1}{2}E_{2}& = & \nabla_{E_{3}}E_{1} \\
  \nabla_{E_{2}}E_{3} & = & \frac{1}{2}E_{1} & = &\nabla_{E_{3}}E_{2}
\end{array}
$$
and $\nabla_{E_{i}}E_{i}=0$ for $i=1,2,3$.

Using the fact that an isometry of $\mathcal{H}_{3}$ which fix the identity, is an automorphism of $\mathfrak{h}_{3}$, it is possible to show that evert isometry of $\mathcal{H}_{3}$ is of the form $L\circ A$ where  $L$ is a left translation in $\mathcal{H}_{3}$ and $A$  is in one of the following forms
$$\begin{bmatrix}
    \cos\theta & -\sin\theta & 0 \\
    \sin\theta & \cos\theta &0 \\
    0 & 0 & 1 \\

  \end{bmatrix}
\quad \textrm{or} \quad  \begin{bmatrix}
     \cos\theta & \sin\theta & 0 \\
    \sin\theta & -\cos\theta &0 \\
    0 & 0 & -1 \\
   \end{bmatrix}
   $$

 That is, $A$ represent a rotation around the $z$-axis or a composition of the reflection across the plane $z=0$ and a reflection across a line $y=mx$ for some $m\in \mathbb{R}$ .

\section{Surfaces in $\mathcal{H}_{3}$}

We will consider two types of surfaces in $\mathcal{H}_{3}$, and calculate the first and second fundamental form in each case.
\subsection{Graph over the xy-plane}

Let $S$ be a graph of a smooth function $f:\Omega \rightarrow \mathbb{R}$ where $\Omega $ is an open set of $\mathbb{R}^{2}$. We consider the
following parametrization of $S,$
\begin{equation}\label{paramet}
 X\left( x,y\right) =( x,y,f( x,y)),\;\;(x,y)\in \Omega.
\end{equation}
A basis of the tangent space $T_{p}S$ associated to this
parametrization is given by
\begin{equation}
\begin{array}{ccccc}
X_{x} & = & \left( 1,0,f_{x}\right) & = & E_{1}+\left( f_{x}+\frac{y}{2}%
\right) E_{3} \\
X_{y} & = & \left( 0,1,f_{y}\right) & = & E_{2}+\left( f_{y}-\frac{x}{2}%
\right) E_{3},%
\end{array}
\label{basis}
\end{equation}%
\noindent
and its unit normal vector is given by
\begin{equation}
\eta \left( x,y\right) =-\left( \frac{f_{x}+\frac{y}{2}}{w}\right)
E_{1}-\left( \frac{f_{y}-\frac{x}{2}}{w}\right) E_{2}+\frac{1}{w}E_{3}
\label{normal}
\end{equation}%
where
\begin{equation}
w=\sqrt{1+\left( f_{x}+\frac{y}{2}\right) ^{2}+\left( f_{y}-\frac{x}{2}\right) ^{2}}.
\label{w}
\end{equation}
Then the
coefficients of the first fundamental form of $S$  are given by%
\begin{equation}
\begin{array}{ccccl}
E & = & <X_{x},X_{x}> & = & 1+\left( f_{x}+\frac{y}{2}\right) ^{2} \\
F & = & <X_{y},X_{x}> & = & \left( f_{x}+\frac{y}{2}\right) \left( f_{y}-%
\frac{x}{2}\right) \\
G & = & <X_{y},X_{y}> & = & 1+\left( f_{y}-\frac{x}{2}\right) ^{2}.%
\end{array}
\label{1ffund}
\end{equation}%
If $\nabla $ is the Riemannian connection of $\left( \mathcal{H}%
_{3},\rd s^{2}\right) $, by the Weingarten  formula for hypersurfaces, we have that%
\[
A_{\eta }v=-\nabla _{v}\eta ,\ \ \ \ v\in T_{p}S
\]%
and the coefficients of the second fundamental form are given by
\begin{equation}
\begin{array}{ccccc}
L & = & -<\nabla _{X_{x}}\eta ,X_{x}> & = & \frac{f_{xx}+( f_{y}-\frac{x%
}{2})( f_{x}+\frac{y}{2})}{w} \\\\
M & = & -<\nabla _{X_{x}}\eta ,X_{y}> & = & \frac{f_{xy}+\frac{1}{2}\left(
f_{y}-\frac{x}{2}\right) ^{2}-\frac{1}{2}\left( f_{x}+\frac{y}{2}\right) ^{2}%
}{w} \\\\
N & = & -<\nabla _{X_{y}}\eta ,X_{y}> & = & \frac{f_{yy}-\left( f_{y}-\frac{x%
}{2}\right) \left( f_{x}+\frac{y}{2}\right) }{w}.%
\end{array}
\label{2ffund}
\end{equation}

\subsection{Vertical surface}

In this case we  consider  such a surface as a ruled surface. We parameterize the surface by

$$X(t,s)=(t,a(t),s),\;\;\;(t,s)\in U$$
where $U$ is an open set of $\mathbb{R}^{2}$. So the basis of the tangent space associated to this parametrization is
\begin{equation}
\begin{array}{ccl}
X_{t} & = & E_{1}+\dot{a}E_{2}+\frac{(a-t\dot{a})}{2}E_{3} \\
X_{s} & = & E_{3}
\end{array}
\label{basisv}
\end{equation}
and the unit normal field to this surface is
$$\eta=\frac{\dot{a}}{\sqrt{1+(\dot{a})^{2}}}E_{1}-\frac{1}{\sqrt{1+(\dot{a})^{2}}}E_{2}$$
The coefficients of the first fundamental forma in the basis $\{X_{t},X_{s}\}$ are given by
\begin{equation}
\begin{array}{ccccl}
E & = & <X_{t},X_{t}> & = & 1+\dot{a}^{2}+\frac{(a-t\dot{a})^{2}}{4} \\\\
F & = & <X_{t},X_{s}> & = &\frac{(a-t\dot{a})}{2} \\\\
G & = & <X_{s},X_{s}> & = & 1.
\end{array}
\label{1ffundv}
\end{equation}

and the coefficients of the second fundamental form are given by
\begin{equation}
\begin{array}{ccccl}
L & = & -<\nabla _{X_{t}}\eta ,X_{t}> & = & \frac{(a-t\dot{a})(1+\dot{a}^{2})-2\ddot{a}}{2\sqrt{1+\dot{a}^{2}}} \\\\
M & = & -<\nabla _{X_{t}}\eta ,X_{s}> & = & \frac{\sqrt{1+\dot{a}^{2}}}{2} \\\\
N & = & -<\nabla _{X_{s}}\eta ,X_{s}> & = & 0.%
\end{array}
\label{2ffundv}
\end{equation}

\section{The Gauss Map}
Recall that the Gauss map is a function  from an oriented surface, $S\subset \mathbb{E}^{3}$, to the unit sphere in the Euclidean space . It associates to every point on the surface its oriented unit normal vector. Considering the Euclidean space as a commutative Lie group, the Gauss map is just the translation of the unit normal vector at any point of the surface to the origin, the identity element of $\mathbb{R}^{3}$. Reasoning in this way we define a Gauss map in the following form:

\begin{definition}
Let $S\subset G$ be an orientable hypersurface of a n-dimensional Lie group $%
G,$ provided with a left invariant metric. The map%
\[
\gamma :S\rightarrow S^{n-1}=\left\{ v\in \tilde{g}:\left\vert v\right\vert
=1\right\}
\]%
where $\gamma \left( p\right) =\rd L_{p}^{-1}\circ \eta \left( p\right) $, $\tilde{g}$ the Lie algebra of $G$
and $\eta $ the unitary normal vector field of $S,$ is called the Gauss map of $S.$
\end{definition}

\noindent We observe that
$$\rd\gamma\left( T_{p}S\right) \subseteq T_{\gamma \left( p\right) }S^{n-1}
 =  \left\{ \gamma \left( p\right) \right\} ^{\perp }  =
\rd L_{p}^{-1}\left( T_{p}S\right),$$

\noindent therefore $\rd L_{p}\circ d\gamma \left( T_{p}S\right)
\subseteq T_{p}S$ .

Now we  obtain a local expression of the Gauss map $\gamma $. In fact, we consider the following sequence of maps
$$\phi:\Omega \overset{X}\longrightarrow X(\Omega)\subset \mathcal{H}_{3}\overset{\gamma}\longrightarrow S^{2}\overset{\psi}\longrightarrow \mathcal{P}$$
where, $X$ is a parametrization of $S$  and $\psi$ is given by $(\ref{psi})$.

If $S$ is a vertical surface, we have
$$X(t,s)=(t,a(t),s),\;\;\;(t,s)\in U$$
In this case, the unit normal to the surface is
$$\eta(t,s)=\frac{\dot{a}}{\sqrt{1+(\dot{a})^{2}}}E_{1}-\frac{1}{\sqrt{1+(\dot{a})^{2}}}E_{2},$$
that is, the image of $\gamma$ is the equator of the sphere. An easy computation shows that the surface is a piece of a vertical plane $Ax+By=C$ if the image of $\gamma$ is a point.

When $S$ is the graph of a smooth function $f\left(x,y\right)$ with $(x,y)$ in a domain $\Omega \subset \mathbb{R}^{2} $.  Then

\begin{equation}\label{gmap}
  \phi(x,y)=\left(-(f_{x}+\frac{y}{2}),-(f_{y}-\frac{x}{2})\right)
\end{equation}
and the Jacobian matrix  of $\phi$ is
\begin{equation}\label{jacob}
\rd \phi_{(x,y)}=\left(
                   \begin{array}{cc}
                     -f_{xx} & -f_{xy}-1/2 \\
                     -f_{xy}+1/2 & -f_{yy} \\
                   \end{array}
                 \right)
\end{equation}

\noindent Notice that
\begin{equation}
\det  \rd \phi_{(x,y)} =f_{xx}f_{yy}-f_{xy}^{2}+\frac{1}{4},  \label{rank}
\end{equation}%
and we will call this expression, the determinant of the Gauss map at the point $(x,y).$ If $\Omega=\mathbb{R}^{2}$, the greatest lower bound of the absolute value
of $\det\rd \phi_{(x,y)}$ is zero. This was proved by A. Borisenko and E.  Petrov in \cite{borisenko2011surfaces}.

We know that in the Euclidean case the differential of the Gauss
map is just the second fundamental form for surfaces in
$\mathbb{R}^{3},$ this fact can be generalized for hypersurfaces
in any Lie group. The following theorem, see \cite{ripoll1991hypersurfaces},
states a relationship between the Gauss map and the extrinsic
geometry of $S.$

\begin{theorem}
\label{gauss}Let $S$ be an orientable hypersurfaces of a Lie group. Then
$$
\rd L_{p}\circ \rd\gamma _{p}\left( v\right) =-\left( A_{\eta }\left( v\right)
+\alpha _{\bar{\eta}}\left( v\right) \right) ,\ \ \ v\in T_{p}S,
$$
where $A_{\eta }$ is the Weingarten operator, $\alpha _{\bar{\eta}}\left(
v\right) =\nabla _{v}\bar{\eta}$ and $\bar{\eta}$ is the left invariant
vector field such that $\eta \left( p\right) =\bar{\eta}\left( p\right).$
\end{theorem}

As a consequence of this theorem we have the following result
\begin{theorem}
\label{vertical}The vertical plane $Ax+By=C$ is the unique connected surface in $%
\mathcal{H}_{3}$ with the property that its Gauss map is constant.
\end{theorem}

\begin{proof}We can prove that there is no graph of a smooth function with constant Gauss map , for details see \cite{figueroa2012gauss}
\end{proof}

To end this section, we study the effect of the isometries of the Heisenberg group $\mathcal{H}_{3}$ on the  Gauss map of a surface.

\begin{theorem}

  Let $S$ be a graph of a smooth function $f:\Omega \rightarrow \mathbb{R}$ where $\Omega $ is an open set of $\mathbb{R}^{2}$ and $\phi:\Omega\rightarrow \mathcal{P}$ its Gauss map, where $X(\Omega)=S$.
\begin{enumerate}
  \item If $\rho_{\theta}:\mathcal{H}_{3}\rightarrow \mathcal{H}_{3}$ is a rotation about the $z$ axis by an angle $\theta$, then the Gauss map of $\rho_{\theta} (S)$ is $r_{\theta}\circ \phi$, where $r_{\theta}:\mathcal{P}\rightarrow \mathcal{P}$ is a rotation about the origin by an angle $\theta$.
  \item If $\sigma:\mathcal{H}_{3}\rightarrow \mathcal{H}_{3}$ is a reflection across the line $ax+by=0$ compound with the reflection about the plan $z=0$  then the gauss map of $\sigma (S)$ is $\tau\circ \phi$, where $\tau:\mathcal{P}\rightarrow \mathcal{P}$ is a reflection across the line $-bx+ay=0$.
\item If $L:\mathcal{H}_{3}\rightarrow \mathcal{H}_{3}$ is a left translation,  then the gauss map of $L(S)$ is $ \phi. $
\end{enumerate}
\end{theorem}

\begin{proof}
Recall that the Gauss map of $S$ is given by
$$\phi(x,y)=-(f_{x}+\frac{y}{2},f_{y}-\frac{x}{2})$$
  \begin{enumerate}

    \item Since $S$ is the graphic of $f$, then $\rho_{\theta}(S)$ is the graphic of the smooth function $h(\overline{x},\overline{y})$ where
$$\overline{x}=x\cos\theta-y\sin\theta ,\qquad \qquad \overline{y}=x\sin\theta+y\cos\theta$$
and
$$h(\overline{x},\overline{y})=f(\overline{x}\cos\theta+\overline{y}\sin\theta,-\overline{x}\sin\theta+\overline{y}\cos\theta)$$
The Gauss map of $\rho_{\theta}(S)$ is given by $\widetilde{\phi}(\overline{x},\overline{y})=-(h_{\overline{x}}+\frac{\overline{y}}{2},  h_{\overline{y}}-\frac{\overline{x}}{2})$, where
\begin{equation}
\nonumber
\begin{array}{rcl}
  h_{\overline{x}}+\frac{\overline{y}}{2} &=& (f_{x}+\frac{y}{2})\cos\theta-(f_{y}-\frac{x}{2})\sin\theta \\\\
 h_{\overline{y}}-\frac{\overline{x}}{2} &=& (f_{x}+\frac{y}{2})\sin\theta-(f_{y}-\frac{x}{2})\cos\theta
\end{array}
\end{equation}
That is $\widetilde{\phi}(\overline{x},\overline{y})=r_{\theta}\circ \phi(x,y).$

    \item In this case,  $\sigma(S)$ is the graphic of the smooth function $h(\overline{x},\overline{y})$ where
$$\overline{x}=x\cos\theta+y\sin\theta ,\qquad \qquad \overline{y}=x\sin\theta-y\cos\theta$$
and
$$h(\overline{x},\overline{y})=-f(\overline{x}\cos\theta+\overline{y}\sin\theta,\overline{x}\sin\theta-\overline{y}\cos\theta)$$
The Gauss map of $\sigma(S)$ is given by $\widetilde{\phi}(\overline{x},\overline{y})=(h_{\overline{x}}+\frac{\overline{y}}{2},  h_{\overline{y}}-\frac{\overline{x}}{2})$, where
\begin{equation}
\nonumber
\begin{array}{rcl}
  h_{\overline{x}}+\frac{\overline{y}}{2} &=& -(f_{x}+\frac{y}{2})\cos\theta-(f_{y}-\frac{x}{2})\sin\theta \\\\
 h_{\overline{y}}-\frac{\overline{x}}{2} &=& -(f_{x}+\frac{y}{2})\sin\theta+(f_{y}-\frac{x}{2})\cos\theta
\end{array}
\end{equation}
That is $\widetilde{\phi}(\overline{x},\overline{y})=-\rho\circ \phi(x,y)$, where $\rho:\mathcal{P}\rightarrow \mathcal{P}$ is a reflection across the line $-bx+ay=0$
    \item Let $L_{(a,b,c)}$ a left translation in $\mathcal{H}_{3},$
$$L_{(a,b,c)}(x,y,f(x,y))=(x+a,y+b,f(x,y)+\frac{ay}{2}-\frac{bx}{2})$$
Then $L_{(a,b,c)}(S)$ is the graphic of $h(\overline{x},\overline{y})$ where
$$\overline{x}=x+a, \qquad \overline{y}=y+b$$ and
$$h(\overline{x},\overline{y})=f(\overline{x}-a,\overline{y}-b)+\frac{a\overline{y}}{2}-\frac{b\overline{x}}{2}+c$$
The Gauss map of $L_{(a,b,c)}(S)$
\begin{equation}
\nonumber
\begin{array}{rcl}
  h_{\overline{x}}+\frac{\overline{y}}{2} &=& (f_{x}+\frac{y}{2}) \\\\
 h_{\overline{y}}-\frac{\overline{x}}{2} &=& (f_{x}+\frac{y}{2})
\end{array}
\end{equation}
That is, $\widetilde{\phi}(\overline{x},\overline{y})=\phi(x,y)$
  \end{enumerate}
\end{proof}

\section{Tension field and mean curvature of surfaces in $\mathcal{H}_{3}$}
In this section we  establishes a relationship between the tension field of the Gauss map and  mean curvature
of a surface in $\mathcal{H}_{3}$.

Firstly, we recall the mean curvature formula of any surface of $\mathcal{H}_{3}$ in terms of the coefficients of their first and second fundamental forms in some parametrization.
\begin{equation}
H=\frac{1}{2}\left( \frac{EN+GL-2FM}{EG-F^{2}}\right).  \label{mean curvatue}
\end{equation}
When the surface is vertical we use the coefficients given in (\ref{1ffundv}) and (\ref{2ffundv}):
$$H=\frac{1}{2}[\frac{\ddot{a}}{(1+\dot{a}^{2})}]$$
If $H=0$, minimal surface, the surface is a vertical plane $Ax+By=C$.

If $H$ is constant and different from zero, we have the following differential equation
$$\dot{a}=-2H(1+\dot{a}^{2})^{3/2}$$
Solving the above equation, we obtain a vertical surface of constant mean curvature $H$, parameterized by
$$X(t,s)=(t,\sqrt{\frac{-2Ht}{1+2Ht}}t,s)$$
where $(t,s)\in I\times \mathbb{R}$

If the surface is  graph of a smooth function $f$, we replace the coefficients given in  $\left( \ref{1ffund}\right) $ and $\left(\ref{2ffund}\right)$
into the mean curvature formula,
\begin{equation}\label{hec}
\frac{\left( 1+q^{2}\right) f_{xx}-2pq f_{xy}+\left( 1+p^{2}\right) f_{yy}}{(1+p ^{2}+q ^{2})^{3/2}}=2H,
\end{equation}
where $p=f_{x}+\frac{y}{2}$ and $q=f_{y}-\frac{x}{2}$.
Unlike the minimal surface case, for graphs of non-zero constant mean curvature, we have a Bernstein type theorem (see \cite{figueroa2012gauss} for more details).

\begin{theorem}
There are no complete graphs of constant mean curvature $H\neq 0.$
\end{theorem}

When $H=0$ we obtain the equation of the minimal graphs in $\mathcal{H}_{3}$
\begin{equation}
\left( 1+q^{2}\right) f_{xx}-2pq f_{xy}+\left( 1+p^{2}\right) f_{yy}=0,  \label{minec}
\end{equation}

This equation appears for the first time in \cite{bekkar1991exemples}. Before presenting some consequences of the above equation, we shall show some examples of complete minimal graphs, that is $\Omega=\mathbb{R}^{2}$, and using formulas $(\ref{gmap}) $ and $\left( \ref{rank}\right) $ to find the image and the rank of their Gauss map.
\begin{example} \label{plan}
As in Euclidean space $\mathbb{E}^{3},$ the plane $f\left( x,y\right)
=ax+by+c$ is a minimal graph of $\mathcal{H}_{3}.$ In this case  $\phi(\mathbb{R}^{2})=\mathbb{R}^{2}$  and the rank of the Gauss map is $2.$
\end{example}

Another minimal graph may be obtained by searching for solutions of Scherk
type, i.e for solutions of the form $f\left( x,y\right) =g\left( x\right)
+h\left( y\right) +\frac{xy}{2}.$ From this method we find, among others,
the following example, see \cite{sari1992surfaces}.

\begin{example}
A surface of saddle type:
$$f\left( x,y\right) =\frac{xy}{2}+k\left[ \ln
\left( y+\sqrt{1+y^{2}}\right) +y\sqrt{1+y^{2}}\right] ,$$ where
$k\in \mathbb{R}$. Notice that this minimal surface is ruled by
affine lines, i.e. translations of 1-parameter subgroups.
$$
\phi(\mathbb{R}^{2})=\left\{
  \begin{array}{ll}
    u=0, & \hbox{\text{if} $k=0$} \\
    \frac{v^{2}}{2k^{2}}-u^{2}=1, & \hbox{\text{if} $k\neq 0$ }
  \end{array}
\right.
$$
 The image of $\phi$ are geodesics of the hyperbolic plane $\mathcal{P}$ and its rank  is 1.
\end{example}

The following example was found by B. Daniel, using a Weierstrass representation,( see \cite{daniel2011gauss} for more details)
\begin{example} \label{dan}
Let $f(x,y)=xh(y)$, where $h(y)=s-\frac{1}{2\coth s}$,  $s$ and $y$ are related by the equation
$y=\coth s-2s, \;\;s>0$. The image of $\phi$ is
$$\phi(\mathbb{R}^{2})=\{(u,v):v>0\}$$
And the determinant of the Gauss map is equal to
$$-\frac{1}{4}(\frac{1}{\coth^{4}s}-1),\;\;\; s>0.$$
So, the rank of its Gauss map is 2.
\end{example}

As we have seen,we have several solutions of $(\ref{minec})$ defined on the entire $xy$-plane.  Unlike the case of Euclidean spaces, where the only complete minimal graphs are linear (Bernstein's  theorem).

On the other hand, is known that the Gauss map of a minimal surface in the Euclidean space is antiholomorphic. For any graph in the Heisenberg group $\mathcal{H}_{3}$, the following  theorem relates the mean curvature and the tension field of the Gauss map of the surface.
\begin{theorem}\label{HG}
Let $S\subset \mathcal{H}_{3}$ be a graph of a smooth function $f:\Omega \rightarrow \mathbb{R}$ where $\Omega $ is an open set of $\mathbb{R}^{2}$. Suppose that $H$ and $\phi$ are the mean curvature and the Gauss map of $S$, respectively. Then  the tension field of the Gauss map $\tau (\phi)=(\tau (\phi^1),\tau (\phi^2))$ satisfy
\[
\begin{array}{ccc}
  \tau (\phi^1)+2wH_x & = & H[\frac{2(f_{y}-\frac{x}{2})}{w}-4\frac{\partial w}{\partial x}+4(f_{x}+\frac{y}{2})Hw^{2}] \\\\
   \tau (\phi^2)+2wH_y& = &H[ \frac{-2(f_{x}+\frac{y}{2})}{w}-4\frac{\partial w}{\partial y}+4(f_{y}-\frac{x}{2})Hw^{2}]
\end{array}
\]
where $w= \sqrt{1+(f_{x}+\frac{y}{2})^{2}+(f_{y}-\frac{x}{2})^{2}}.$
\end{theorem}
\begin{proof}

   We begin by recalling that the Gauss map of $S$  is given by  $\phi(x,y)=-(f_{x}+y/2,f_{y}-x/2)$ and the metric tensors of $S$  is $g=(g_{ij})$. By definition (see for instance \cite{eells1978report}), the tension field of $\phi$ is given by
  $$\tau(\phi^\alpha)=\triangle\phi^{\alpha}+\Gamma'^{\alpha}_{\beta\gamma}\phi^{\beta}_{i}\phi^{\gamma}_{j}g^{ij}$$

 \noindent where $\alpha,\beta=1,2$,  $\phi^{1}=-(f_{x}+y/2)$, $\phi^{2}=-(f_{y}-x/2)$,
$\triangle$ is the Riemannian Laplace on $S$, $\Gamma'^{\alpha}_{\beta\gamma}$ are the Christoffel symbols of the metric $h$ of $\mathcal{P}$, evaluated in the Gauss map and  $\phi^{\beta}_{1}=\frac{\partial\phi^{\beta}}{\partial x}$ and $\phi^{\beta}_{2}=\frac{\partial\phi^{\beta}}{\partial y}$.

  First we compute the first component of the tension field and add twice the $x$-derivative of $wH$. See (\ref{hec}) and (\ref{T1Hu}) from  Appendix \ref{appendix:a}.

  \begin{equation*}
\begin{split}
\tau (\phi^{1})+2(Hw)_{x} &= 4[(1+(f_{y}-\frac{x}{2})^{2})f_{xx}-2( f_x+\frac{y}{2}) (f_y -\frac{x}{2} )f_{xy}+(1+(f_{x}+\frac{y}{2})^{2})f_{yy}] \\
&\quad [(2 f_x+y)( (x-2 f_y)^{2}f_{xx}+2 (2 f_x+y)(x-2 f_y)f_{xy}+  \\
&\quad(4+(2f_{x}+y)^{2})f_{yy})+(4 f_{xy}-6)(x-2 f_y)]/32w^{4}
\end{split}
\end{equation*}

Substituting the mean curvature (\ref{hec})  in the above equation gives

\begin{equation*}
\begin{split}
\tau (\phi^{1})+2(wH)_{x} &= 8Hw^{3}[(2 f_x+y)( (x-2 f_y)^{2}f_{xx}+2 (2 f_x+y)(x-2 f_y)f_{xy}+\\ &\quad(4+(2f_{x}+y)^{2})f_{yy}) +(4 f_{xy}-6)(x-2 f_y)]/32w^{4}
\end{split}
\end{equation*}

\noindent Replacing again  the mean curvature in the above equation, we get

$$\tau (\phi^{1})+2(wH)_{x}=\frac{H}{4w}[-4(x-2f_y)-(4+(2f_y-x)^2+(y+2f_x)^2)_{x}+8(y+2f_x)Hw^{3}]$$
Since $4w^2=4+(2f_y-x)^2+(y+2f_x)^2$  we have
$$\tau (\phi^{1})+2(wH)_{x}=\frac{H}{4w}[-4(x-2f_y)-(4w^{2})_x+8(y+2f_x)Hw^{3}]$$
Finally using the equality $2(wH)_x=2w_xH+2wH_x$ we obtain the desired formula.

We now apply (\ref{T2Hv}) from Appendix \ref{appendix:a} and the above argument to obtain the second formula of the theorem, which completes the proof.

\end{proof}

  Let us mention an important consequence of the above theorem.
 \begin{corollary}
 If $S\subset \mathcal{H}_{3}$ is a minimal surface, then the Gauss map of $S$ is harmonic.
 \end{corollary}
 \begin{proof}

 If $S$ is the graph of a smooth function $f:\Omega\rightarrow \mathbb{R}$ with $H=0$. We conclude from Theorem \ref{HG} that $\tau(\phi)=0$,
 and, in consequence, $\phi$ is harmonic.

 If $S$ is a minimal vertical surface, then $S$ is a vertical plane and its Gauss map is constant, which implies the Gauss map is harmonic.
 \end{proof}

The above corollary together with the next result, will allow us to study the minimal surfaces in $\mathcal{H}_{3}$
\begin{theorem}
Let $M$ and $N$ two riemannian manifolds such that $M$ is connected. If $F:M \rightarrow N$ is harmonic and $\rd f$ has rank $1$ in an open set, then
$F$ maps $M$ into a geodesic arc in $N$
\end{theorem}
\begin{proof}For the proof we refer the reader to  \cite{sampson1978some}
\end{proof}

\subsection{Minimal graphs of rank 1}
Since the only  minimal vertical surface is a vertical plane and its Gauss map is constant we consider only  minimal graphs whose Gauss map  has rank 1.  Let $S$ be a graph of a smooth function $f:\Omega \rightarrow \mathbb{R}$ where $\Omega $ is a domain of $\mathbb{R}^{2}$. We consider the following parametrization of $S,$

\begin{equation}\label{domain}
   X\left( x,y\right) =( x,y,f( x,y)),(x,y)\in \Omega.
\end{equation}
and $\phi:\Omega\rightarrow \mathcal{P}$ its Gauss map, see $(\ref{gmap})$. Since $\phi$ is harmonic, it follows that $\phi(\Omega)$ is a geodesic in $\mathcal{P}$.
From this we have two cases:

\begin{description}
  \item[$\phi(\Omega)$ is a straight line going through the origin] . We can assume, by rotating an suitable angle in  $\mathcal{H}_{3}$, that
$$f_{y}-x/2=0$$
in $\Omega$. Thus, $f(x,y)=xy/2 +h(x)$. Substituting into the minimal graph equation $(\ref{minec})$ we obtain $h''(x)=0$. Therefore
$$f(x,y)=\frac{xy}{2}+kx+c.
$$
  \item[$\phi(\Omega)$ is a branch of a hyperbola.] As in the previous case, we can assume, by rotating an suitable angle  in  $\mathcal{H}_{3}$, that $\phi(\Omega)$ is a branch of the following hyperbola,
$$\frac{(f_{y}-x/2)^{2}}{a^{2}}-\frac{(f_{x}+y/2)^{2}}{b^{2}}=1$$
such that, $(f_{x}+y/2)(0,0)=0$ and $(f_{y}-x/2)(0,0)\neq 0$.
Differentiating the hyperbola equation with respect to $y$ and evaluating at $(0,0)$ we obtain $f_{yy}(0,0)=0$. From this we conclude that
$$f(x,y)=\frac{xy}{2}+k[\ln (y+\sqrt{1+y^{2}})+y\sqrt{1+y^{2}}]+c$$
This can be found in lemma $(14)$ of \cite{figueroa2012gauss}.
\end{description}

\subsection{Minimal graphs of rank 2}
In this section, we will characterize the minimal graphs in $\mathcal{H}_{3}$ whose Gauss map is conformal. More precisely

\begin{theorem}
  Let $S\subset \mathcal{H}_{3}$ be a minimal graph. $S$ is a plane iff its Gauss map is conformal.
\end{theorem}
\begin{proof}
Let $S$ parameterized as in $(\ref{domain})$. By  (\ref{gmap}), the Gauss map of $S$ is given by $\phi(x,y)=\left(-p(x,y),-q(x,y)\right),$
where
$$p(x,y)=(f_{x}+\frac{y}{2}),\;\;\;q(x,y)=(f_{y}-\frac{x}{2}).$$
  Consider $\phi$ conformal, that is, $\rd\phi_{(x,y)}\neq 0$ for all $(x,y)\in \Omega$ and
  $$<\rd \phi(\textbf{u}),\rd \phi(\textbf{v})>_{\phi(x,y)}=\lambda(x,y)<\rd X (\textbf{u}),\rd X (\textbf{v})>_{X(x,y)},$$
  for all tangents vector $\textbf{u,v}$ at $(x,y)\in \Omega$ and $\lambda>0$. In particular, if $\{\frac{\partial}{\partial x},\frac{\partial}{\partial y}\}$ is the canonical basis on $\Omega $, then
  $$\lambda(x,y)=\frac{<\rd \phi(\frac{\partial}{\partial x}),\rd \phi(\frac{\partial}{\partial x})>}{<X_{x},X_{x}>}=\frac{<\rd \phi(\frac{\partial}{\partial y}),\rd \phi(\frac{\partial}{\partial y})>}{<X_{y},X_{y}>}$$
  where $\rd\phi(\frac{\partial}{\partial x})$ and $\rd \phi(\frac{\partial}{\partial y})$ are nothing but the columns of $\rd \phi_{(x,y)}$, see (\ref{jacob}).
  Combining this equation with the minimal graph equation $(\ref{minec})$, we obtain
  $$pq\left((1+q^{2})f_{xx}+(1+p^{2})f_{yy}\right)-2(1+q^{2})(1+p^{2})f_{xy}=0$$
  Applying (\ref{minec}) again, we conclude that
$$2(1+p^{2}+q^{2})f_{xy}=0,$$
that is,  $f_{xy}=0$ on  $\Omega$.

In the same manner we can see that

$$<\rd \phi(\frac{\partial}{\partial x}),\rd \phi(\frac{\partial}{\partial x})><X_{x},X_{y}>=<\rd \phi(\frac{\partial}{\partial x}),\rd \phi(\frac{\partial}{\partial y})><X_{x},X_{x}>,$$
 Applying the minimal graph equation $(\ref{minec})$ and considering that $f_{xy}=0$, we obtain,
$$(1+p^{2}+q^{2})f_{xx}=0.$$
 That is, $f_{xx}=0$ and by  $(\ref{minec})$ it is obvious that $f_{yy}=0$ on $\Omega$.

On the contrary, if $S$ is the plane $z=0$. We consider a parametrization $X(x,y)=(x,y,0)$ and its   Gauss map is given by
$\phi(x,y)=(-y/2,x/2)$. It is easy to check that $\phi$ is conformal.

\end{proof}

\appendix
\section{Tension field}
\label{appendix:a}

This is a Mathematica program to compute the Tension Field of the Gauss map $\phi:S\rightarrow \mathcal{P}$ where $S$ is the graph of a function $f:\Omega\rightarrow \mathbb{R}$ and $\mathcal{P}$ is the Gans model of the hyperbolic plane.

We need the partial derivatives of $f$ up to order three:

\begin{doublespace}
\noindent\(\pmb{\text{Clear}[u,v,x,y]}\)
\end{doublespace}

\begin{doublespace}
\noindent\(\pmb{f[\text{u$\_$},\text{v$\_$}];f_x[\text{u$\_$},\text{v$\_$}];f_y[\text{u$\_$},\text{v$\_$}];f_{\text{xx}}[\text{u$\_$},\text{v$\_$}];f_{\text{xy}}[\text{u$\_$},\text{v$\_$}];f_{\text{yx}}[\text{u$\_$},\text{v$\_$}];}\\
\pmb{f_{\text{yy}}[\text{u$\_$},\text{v$\_$}];f_{\text{xyy}}[\text{u$\_$},\text{v$\_$}];f_{\text{yyy}}[\text{u$\_$},\text{v$\_$}];f_{\text{xxx}}[\text{u$\_$},\text{v$\_$}];f_{\text{xxy}}[\text{u$\_$},\text{v$\_$}];}\)
\end{doublespace}

\begin{doublespace}
\noindent\(\pmb{\text{Derivative}[1,0][f][u,v]\text{:=}f_x[u,v]}\)
\end{doublespace}

\begin{doublespace}
\noindent\(\pmb{\text{Derivative}[0,1][f][u,v]\text{:=}f_y[u,v]}\)
\end{doublespace}

\begin{doublespace}
\noindent\(\pmb{\text{Derivative}[1,0]\left[f_x\right][u,v]\text{:=}f_{\text{xx}}[u,v]}\)
\end{doublespace}

\begin{doublespace}
\noindent\(\pmb{\text{Derivative}[0,1]\left[f_x\right][u,v]\text{:=}f_{\text{xy}}[u,v]}\)
\end{doublespace}

\begin{doublespace}
\noindent\(\pmb{\text{Derivative}[1,0]\left[f_y\right][u,v]\text{:=}f_{\text{xy}}[u,v]}\)
\end{doublespace}

\begin{doublespace}
\noindent\(\pmb{\text{Derivative}[0,1]\left[f_y\right][u,v]\text{:=}f_{\text{yy}}[u,v]}\)
\end{doublespace}

\begin{doublespace}
\noindent\(\pmb{\text{Derivative}[0,1]\left[f_{\text{xy}}\right][u,v]\text{:=}f_{\text{xyy}}}\)
\end{doublespace}

\begin{doublespace}
\noindent\(\pmb{\text{Derivative}[1,0]\left[f_{\text{yy}}\right][u,v]\text{:=}f_{\text{xyy}}}\)
\end{doublespace}

\begin{doublespace}
\noindent\(\pmb{\text{Derivative}[0,1]\left[f_{\text{yy}}\right][u,v]\text{:=}f_{\text{yyy}}}\)
\end{doublespace}

\begin{doublespace}
\noindent\(\pmb{\text{Derivative}[1,0]\left[f_{\text{xx}}\right][u,v]\text{:=}f_{\text{xxx}}}\)
\end{doublespace}

\begin{doublespace}
\noindent\(\pmb{\text{Derivative}[0,1]\left[f_{\text{xx}}\right][u,v]\text{:=}f_{\text{xxy}}}\)
\end{doublespace}

\begin{doublespace}
\noindent\(\pmb{\text{Derivative}[1,0]\left[f_{\text{xy}}\right][u,v]\text{:=}f_{\text{xxy}}}\)
\end{doublespace}

\begin{doublespace}
\noindent\(\pmb{X[\text{u$\_$},\text{v$\_$}]\text{:=}\{u,v,f[u,v]\}}\)
\end{doublespace}

Let $\{$\(X_x\),\(X_y\)$\}$ be the basis of the tangent space \(T_pS\) associated to the parametrization $X$. The components of \(X_x\) and
\(X_y\) with respect to the orthonormal basis of the Heisenberg group, ${E_i}$ with ${i=1,2,3}$, are:

\begin{doublespace}
\noindent\(\pmb{X_x=\left\{1,0,f_x[u,v]+\frac{v}{2}\right\}}\)
\end{doublespace}

\begin{doublespace}
\noindent\(\left\{1,0,\frac{v}{2}+f_x[u,v]\right\}\)
\end{doublespace}

\begin{doublespace}
\noindent\(\pmb{X_y=\left\{0,1,f_y[u,v]-\frac{u}{2}\right\}}\)
\end{doublespace}

\begin{doublespace}
\noindent\(\left\{0,1,-\frac{u}{2}+f_y[u,v]\right\}\)
\end{doublespace}

\noindent The normal field to the surface $S$ is given by

\begin{doublespace}
\noindent\(\pmb{\text{N1}=\text{Cross}\left[X_x,X_y\right]}\)
\end{doublespace}

\begin{doublespace}
\noindent\(\left\{-\frac{v}{2}-f_x[u,v],\frac{u}{2}-f_y[u,v],1\right\}\)
\end{doublespace}

\noindent The Gauss map $\phi $:S$\rightarrow \mathcal{P}$ is given by:

\begin{doublespace}
\noindent\(\pmb{\text{$\phi $1}[\text{u$\_$},\text{v$\_$}]\text{:=}-\left(\frac{v}{2}+f_x[u,v]\right)}\)
\end{doublespace}

\begin{doublespace}
\noindent\(\pmb{\text{$\phi $2}[\text{u$\_$},\text{v$\_$}]\text{:=}-\left(f_y[u,v]-\frac{u}{2}\right)}\)
\end{doublespace}

\begin{doublespace}
\noindent\(\pmb{\phi =\{\text{$\phi $1}[u,v],\text{$\phi $2}[u,v]\}}\)
\end{doublespace}

\begin{doublespace}
\noindent\(\left\{-\frac{v}{2}-f_x[u,v],\frac{u}{2}-f_y[u,v]\right\}\)
\end{doublespace}

\noindent The coefficients of the first fundamental form of the graph of $f$ are:

\begin{doublespace}
\noindent\(\pmb{\text{g11}=\text{Dot}\left[X_x,X_x\right]}\)
\end{doublespace}

\begin{doublespace}
\noindent\(1+\left(\frac{v}{2}+f_x[u,v]\right){}^2\)
\end{doublespace}

\begin{doublespace}
\noindent\(\pmb{\text{g12}=\text{Dot}\left[X_x,X_y\right]}\)
\end{doublespace}

\begin{doublespace}
\noindent\(\left(\frac{v}{2}+f_x[u,v]\right) \left(-\frac{u}{2}+f_y[u,v]\right)\)
\end{doublespace}

\begin{doublespace}
\noindent\(\pmb{\text{g21}=\text{Dot}\left[X_y,X_x\right]}\)
\end{doublespace}

\begin{doublespace}
\noindent\(\left(\frac{v}{2}+f_x[u,v]\right) \left(-\frac{u}{2}+f_y[u,v]\right)\)
\end{doublespace}

\begin{doublespace}
\noindent\(\pmb{\text{g22}=\text{Dot}\left[X_y,X_y\right]}\)
\end{doublespace}

\begin{doublespace}
\noindent\(1+\left(-\frac{u}{2}+f_y[u,v]\right){}^2\)
\end{doublespace}

\noindent This is a \textit{ Mathematica} program to compute the Riemannian Laplacian of the Gauss map. The Laplacian of a differential function $h:S\rightarrow \mathbb{R}$
is calculated by the formula

$$\Delta h=\frac{1}{\sqrt{\det(g)}}\sum _{i,j=2}^2 \partial _{x_i}\left(\sqrt{\det(g)}g^{ij} \partial _{x^j}h\right)$$

\noindent where $(g^{ij})$ is the inverse matrix of the metric $(g_{ij})$.

\begin{doublespace}
\noindent\(\pmb{\text{Clear}[\text{coord}, \text{metric},\text{inversemetric}]}\)
\end{doublespace}

\begin{doublespace}
\noindent\(\pmb{\pmb{\text{coord} = \{u,v\};}}\)
\end{doublespace}

\noindent We input the metric of the surface $S$ as a matrix.

\begin{doublespace}
\noindent\(\pmb{\text{metric}=\{\{\text{g11},\text{g21}\},\{\text{g12},\text{g22}\}\};}\)
\end{doublespace}

\noindent In matrix form:

\begin{doublespace}
\noindent\(\pmb{\text{metric}\text{//}\text{MatrixForm};}\)
\end{doublespace}

\noindent The inverse metric is obtained through matrix inversion.

\begin{doublespace}
\noindent\(\pmb{\text{inversemetric}=\text{Inverse}[\text{metric}];}\)
\end{doublespace}

\noindent The inverse metric can also be displayed in matrix form:

\begin{doublespace}
\noindent\(\pmb{\text{inversemetric}\text{//}\text{MatrixForm};}\)
\end{doublespace}

\noindent The Laplacian of the first component of the map Gauss, $\phi ^1$

\begin{doublespace}
\noindent\(\pmb{\text{LA1}=\text{Simplify}\left[\frac{1}{\sqrt{\text{Det}[\text{metric}]}}*\text{Sum}\left[D\left[\sqrt{\text{Det}[\text{metric}]}*\text{inversemetric}[[i,j]]*D[\phi
[[1]],\right.\right.\right.}\\
\pmb{\text{coord}[[i]] ],\text{coord}[[j]]],\{i,1,2\},\{j,1,2\}]];}\)
\end{doublespace}

\noindent The Laplacian of the second component of the map Gauss, $\phi ^2$

\begin{doublespace}
\noindent\(\pmb{\text{LA2}=\text{Simplify}\left[\frac{1}{\sqrt{\text{Det}[\text{metric}]}}*\text{Sum}\left[D\left[\sqrt{\text{Det}[\text{metric}]}*\text{inversemetric}[[i,j]]*D[\phi
[[2]],\right.\right.\right.}\\
\pmb{\text{coord}[[i]] ],\text{coord}[[j]]],\{i,1,2\},\{j,1,2\}]];}\)
\end{doublespace}

\noindent This is a \textit{ Mathematica} program to compute the Christoffel symbols and was adapted from the notebook \textit{ Curvature and the Einstein }equation  written by \textit{ Leonard Parker }, see \cite{hartle2003gravity}. The Christoffel symbols are calculated by the formula

$$\Gamma ^{\lambda }_{\mu \nu }=\frac{1}{2}g^{\lambda \sigma }\left(\partial _{\mu } g_{\sigma \nu }+\partial _{\nu } g_{\sigma \mu }-\partial
_{\sigma } g_{\mu \nu }\right)$$

\noindent where $(g^{\lambda \sigma })$ is the inverse of the matrix of the metric $g_{\lambda \sigma}$.

\begin{doublespace}
\noindent\(\pmb{\text{Clear}[\text{coord}, \text{metricgg},\text{inversemetricgg}, \text{affine},\text{  }u,v]}\)
\end{doublespace}

\begin{doublespace}
\noindent\(\pmb{\pmb{\text{coord} = \{u,v\};}}\)
\end{doublespace}

\noindent We input the metric of the Gans model of the Hyperbolic space $\mathcal{P}$ as a matrix.

\begin{doublespace}
\noindent\(\pmb{\text{metricgg}=\left\{\left\{\frac{1+v^2}{1+u^2+v^2},\frac{-u v}{1+u^2+v^2}\right\},\left\{\frac{-u v}{1+u^2+v^2},\frac{1+u^2}{1+u^2+v^2}\right\}\right\};}\)
\end{doublespace}

\noindent In matrix form:

\begin{doublespace}
\noindent\(\pmb{\text{metricgg}\text{//}\text{MatrixForm};}\)
\end{doublespace}

\noindent The inverse metric is obtained through matrix inversion.

\begin{doublespace}
\noindent\(\pmb{\text{inversemetricgg}=\text{Simplify}[\text{Inverse}[\text{metricgg}]];}\)
\end{doublespace}

\noindent The inverse metric can also be displayed in matrix form:

\begin{doublespace}
\noindent\(\pmb{\text{inversemetricgg}\text{//}\text{MatrixForm};}\)
\end{doublespace}

\noindent The calculation of the components of the affine connection is done by transcribing the definition given earlier into the notation of \textit{ Mathematica}
and using the \textit{ Mathematica} functions \pmb{ D} for taking partial derivatives, \pmb{ Sum} for summing over repeated indices, \pmb{ Table}
for forming a list of components, and \pmb{ Simplify} for simplifying the result.

\begin{doublespace}
\noindent\(\pmb{\text{affine}\text{:=}\text{affine}=\text{Simplify}[\text{Table}[(1/2)*\text{Sum}[(\text{inversemetricgg}[[i,s]])*}\\
\pmb{(D[\text{metricgg}[[s,j]],\text{coord}[[k]]]+}\pmb{D[\text{metricgg}[[s,k]],\text{coord}[[j]] ]-}\\
\pmb{ D[\text{metricgg}[[j,k]],\text{coord}[[s]] ]),\{s,1,2\}],}\\
\pmb{\{i,1,2\},\{j,1,2\},\{k,1,2\}] ]}\)
\end{doublespace}

\noindent The components of the affine connections of the Gans model are displayed below. Because the affine connection is symmetric under interchange of the
last two indices, only the independent components are displayed.

\begin{doublespace}
\noindent\(\pmb{\text{listaffine}\text{:=}\text{Table}[\text{If}[\text{UnsameQ}[\text{affine}[[i,j,k]],0],\{\text{ToString}[\Gamma [i,j,k]],\text{affine}[[i,j,k]]\}]
,}\\
\pmb{\{i,1,2\},\{j,1,2\},\{k,1,j\}]}\)
\end{doublespace}

\begin{doublespace}
\noindent\(\pmb{\text{TableForm}[\text{Partition}[\text{DeleteCases}[\text{Flatten}[\text{listaffine}],\text{Null}],2],\text{TableSpacing}\to \{2,2\}]}\)
\end{doublespace}

\begin{doublespace}
\noindent\(\begin{array}{ll}
 \text{$\Gamma $[1, 1, 1]} & -\frac{u \left(1+v^2\right)}{1+u^2+v^2} \\
 \text{$\Gamma $[1, 2, 1]} & \frac{u^2 v}{1+u^2+v^2} \\
 \text{$\Gamma $[1, 2, 2]} & -\frac{u+u^3}{1+u^2+v^2} \\
 \text{$\Gamma $[2, 1, 1]} & -\frac{v+v^3}{1+u^2+v^2} \\
 \text{$\Gamma $[2, 2, 1]} & \frac{u v^2}{1+u^2+v^2} \\
 \text{$\Gamma $[2, 2, 2]} & -\frac{\left(1+u^2\right) v}{1+u^2+v^2} \\
\end{array}\)
\end{doublespace}

\noindent In matrix form:

\begin{doublespace}
\noindent\(\pmb{ \text{l1}=\{\{\text{affine}[[1,1,1]],\text{affine}[[1,1,2]]\},\{\text{affine}[[1,2,1]],\text{affine}[[1,2,2]]\}\};}\)
\end{doublespace}

\begin{doublespace}
\noindent\(\pmb{ \text{l2}=\{\{\text{affine}[[2,1,1]],\text{affine}[[2,1,2]]\},\{\text{affine}[[2,2,1]],\text{affine}[[2,2,2]]\}\};}\)
\end{doublespace}

\noindent This is a \textit{ Mathematica} program to compute the Tension field of the Gauss map. The tension field is calculated by the formula

$$\tau ( \phi ^{\alpha })=\Delta \phi ^{\alpha }+
\sum _{i,j=1}^2 \sum _{\beta ,\gamma =1}^2 \Gamma _{\beta \gamma }^{\alpha }\frac{\partial \phi ^{\beta }}{\partial i}\frac{\partial \phi ^{\gamma }}{\partial j}g^{ij}$$

Where $\Delta\phi^{\alpha}$ is the S-Laplacian of the component $\phi^{\alpha}$, $\Gamma_{\beta\gamma}^{\alpha}$ the Christoffel symbols of the Hyperbolic space evaluate at the Gauss map, $\frac{\partial \phi ^{\beta }}{\partial 1}$ is the u-derivative of $\phi ^{\beta }$ and $\frac{\partial \phi ^{\beta }}{\partial 2},$ the v-derivative of $\phi ^{\beta }.$

\noindent First, we evaluate the Christoffel symbols of the Gans model at the Gauss map $\phi $:

\begin{doublespace}
\noindent\(\pmb{\text{AFF1}=\text{l1}\text{/.}\{u\to \phi [[1]],v\to \phi [[2]]\};}\)
\end{doublespace}

\begin{doublespace}
\noindent\(\pmb{\text{AFF2}=\text{l2}\text{/.}\{u\to \phi [[1]],v\to \phi [[2]]\};}\)
\end{doublespace}

\noindent Let D1=\(\sum _{i,j=1}^2 \sum _{\beta ,\gamma =1}^2 \Gamma _{\beta \gamma }^1\)\(\partial _i\phi ^{\beta }\)\(\partial _j\phi ^{\gamma }\)\(g^{\text{ij}}\)

\begin{doublespace}
\noindent\(\pmb{\text{D1}=\text{Sum}[\text{inversemetric}[[i,j]]*\text{AFF1}[[k,l]] * D[\phi [[k]],\text{coord}[[i]]]*D[\phi [[l]],}\\
\pmb{\text{coord}[[j]]],\{i,1,2\},\{j,1,2\},\{k,1,2\},\{l,1,2\}];}\)
\end{doublespace}

\begin{doublespace}
\noindent\(\pmb{\text{DD1}=\text{Simplify}[\text{D1}]; }\)
\end{doublespace}

\noindent Let D2=\(\sum _{i,j=1}^2 \sum _{\beta ,\gamma =1}^2 \Gamma _{\beta \gamma }^2\)\(\partial _i\phi ^{\beta }\)\(\partial _j\phi ^{\gamma }\)\(g^{\text{ij}}\)

\begin{doublespace}
\noindent\(\pmb{\text{D2}=\text{Sum}[\text{inversemetric}[[i,j]]*\text{AFF2}[[k,l]] * D[\phi [[k]],\text{coord}[[i]]]*D[\phi [[l]],}\\
\pmb{\text{coord}[[j]]],\{i,1,2\},\{j,1,2\},\{k,1,2\},\{l,1,2\}];}\)
\end{doublespace}

\begin{doublespace}
\noindent\(\pmb{\text{DD2}=\text{Simplify}[\text{D2}];}\)
\end{doublespace}

\noindent The first component of the tension field of the map Gauss, $\tau (\phi ^1 )$ is given by

\begin{doublespace}
\noindent\(\pmb{\text{Ten1}=\text{Simplify}[\text{LA1}+\text{DD1}];}\)
\end{doublespace}

\noindent In the same manner we can see the second component of the tension field of the Gauss map, $\tau (\phi ^2)$ is given by

\begin{doublespace}
\noindent\(\pmb{\text{Ten2}=\text{Simplify}[\text{LA2}+\text{DD2}];}\)
\end{doublespace}

\noindent The x-derivative of the mean curvature equation is:

\begin{doublespace}
\noindent\(\pmb{\text{MU}=D\left[\left(\left(1+\left(f_y[u,v]-\frac{u}{2}\right){}^2\right)f_{\text{xx}}[u,v]-2\left(f_x[u,v]+\frac{v}{2}\right)\left(f_y[u,v]-\frac{u}{2}\right)f_{\text{xy}}[u,v]+\right.\right.}\\
\pmb{\left.\left.\left(1+\left(f_x[u,v]+\frac{v}{2}\right){}^2\right)f_{\text{yy}}[u,v]\right)/\left(1+\left(f_x[u,v]+\frac{v}{2}\right){}^2+\left(f_y[u,v]-\frac{u}{2}\right){}^2\right),\{u,1\}\right];}\)
\end{doublespace}

\begin{doublespace}
\noindent\(\pmb{\text{FullSimplify}[\text{Ten1}+\text{MU}];}\)
\end{doublespace}

\noindent The simplify expression of $\tau (\phi ^1 )+2(wH)_x $ is:

\begin{doublespace}
\noindent\(\pmb{\text{ReplaceAll}\left[\%,\left\{f_x[u,v]\to f_x,f_y[u,v]\to f_y,f_{\text{xx}}[u,v]\to f_{\text{xx}},f_{\text{xy}}[u,v]\to f_{\text{xy}},\right.\right.}\\
\pmb{\left.\left. f_{\text{yy}}[u,v]\to f_{\text{yy}}\right\}\right]}\)
\end{doublespace}

The numerator of the above expression is:

\begin{doublespace}
\noindent\(\pmb{\text{Numerator}[\%]}\)
\end{doublespace}

\begin{multline}\label{T1Hu}
  [(f_{\text{xx}} (4+u^2+4 f_y(-u+f_y))+
4 f_{\text{yy}}+(v+2 f_x) (2 f_{\text{xy}} (u-2 f_y)+(v+2 f_x) f_{\text{yy}})] \\
[(-6+4 f_{\text{xy}}+(v+2 f_x) (2 (v+2 f_x) f_{\text{xy}}+f_{\text{xx}} (u-2 f_y)))(u-2 f_y)+\\
(v+2 f_x) (4+v^2+4 f_x (v+f_x)) f_{\text{yy}}]
\end{multline}

\noindent We now apply this argument again, with u replaced with v, to obtain the y-derivative of the mean curvature equation

\begin{doublespace}
\noindent\(\pmb{\text{MV}=D\left[\left(\left(1+\left(f_y[u,v]-\frac{u}{2}\right){}^2\right)f_{\text{xx}}[u,v]-2\left(f_x[u,v]+\frac{v}{2}\right)\left(f_y[u,v]-\frac{u}{2}\right)f_{\text{xy}}[u,v]+\right.\right.}\\
\pmb{\left.\left.\left(1+\left(f_x[u,v]+\frac{v}{2}\right){}^2\right)f_{\text{yy}}[u,v]\right)/\left(1+\left(f_x[u,v]+\frac{v}{2}\right){}^2+\left(f_y[u,v]-\frac{u}{2}\right){}^2\right),\{v,1\}\right];}\)
\end{doublespace}

\begin{doublespace}
\noindent\(\pmb{\text{FullSimplify}[\text{Ten2}+\text{MV}];}\)
\end{doublespace}

\noindent The simplify expression of $\tau (\phi ^2)+2(wH)_y$

\begin{doublespace}
\noindent\(\pmb{\text{ReplaceAll}\left[\%,\left\{f_x[u,v]\to f_x,f_y[u,v]\to f_y,f_{\text{xx}}[u,v]\to f_{\text{xx}},f_{\text{xy}}[u,v]\to f_{\text{xy}},\right.\right.}\\
\pmb{\left.\left. f_{\text{yy}}[u,v]\to f_{\text{yy}}\right\}\right]}\)
\end{doublespace}

The numerator of the above expression is:

\begin{doublespace}
\noindent\(\pmb{\text{Numerator}[\%]}\)
\end{doublespace}
\begin{multline}\label{T2Hv}
[-(f_{\text{xx}} (4+u^2+4 f_y (-u+f_y))+4 f_{\text{yy}}+(v+2 f_x) (2 f_{\text{xy}} (u-2
f_y)+(v+2 f_x) f_{\text{yy}})] \\
[2 v (3+2 f_{\text{xy}})+4 f_x^2 (u-2 f_y) f_{\text{yy}}+4f_x (3+f_{\text{xy}} (2+u^2+4 f_y (-u+f_y))+\\
v (u-2 f_y) f_{\text{yy}})+(u-2 f_y) (f_{\text{xx}}
(4+u^2+4 f_y (-u+f_y))+v (2 f_{\text{xy}} (u-2 f_y)+v f_{\text{yy}})))]
\end{multline}

\end{document}